\documentclass[12pt,fleqn]{article}

\usepackage{amsmath,amssymb,amsthm}
\usepackage{geometry}
\usepackage[pdfpagemode=UseNone,pdfstartview=FitH,hypertex]{hyperref}

\newcommand{\bdry}[1]{\partial #1}
\newcommand{\bgset}[1]{\big\{#1\big\}}
\newcommand{\A}{{\cal A}}
\newcommand{\D}{{\cal D}}
\newcommand{\closure}[1]{\overline{#1}}
\newcommand{\dint}{\ds{\int}}
\newcommand{\dist}[2]{\text{dist}\, (#1,#2)}
\newcommand{\ds}[1]{\displaystyle #1}
\newcommand{\eps}{\varepsilon}
\newcommand{\M}{{\cal M}}
\newcommand{\norm}[2][]{\left\|#2\right\|_{#1}}
\renewcommand{\o}{\text{o}}
\newcommand{\PS}[1]{$(\text{PS})_{#1}$}
\newcommand{\pnorm}[2][]{\if #1'' \left|#2\right|_p \else \left|#2\right|_{#1} \fi}
\newcommand{\QED}{\mbox{\qedhere}}
\newcommand{\R}{\mathbb R}
\newcommand{\restr}[2]{\left.#1\right|_{#2}}
\newcommand{\seq}[1]{\left(#1\right)}
\newcommand{\set}[1]{\left\{#1\right\}}
\newcommand{\spnorm}[2][]{\if #1'' |#2|_p \else |#2|_{#1} \fi}
\newcommand{\vol}[1]{\left|#1\right|}

\DeclareMathOperator*{\essinf}{ess\, inf}

\newtheorem{corollary}{Corollary}[section]
\newtheorem{lemma}[corollary]{Lemma}
\newtheorem{proposition}[corollary]{Proposition}
\newtheorem{theorem}[corollary]{Theorem}

\numberwithin{equation}{section}

\title{\bf Multiplicity results for critical $p$-Laplacian problems}
\author{\bf Giuseppina Barletta \& Pasquale Candito\\
Universit\`{a} degli Studi di Reggio Calabria\\
89100 Reggio Calabria, Italy\\
\em giuseppina.barletta@unirc.it \& pasquale.candito@unirc.it\\
[\bigskipamount]
\bf Salvatore A. Marano\\
Universit\`{a} degli Studi di Catania\\
95125 Catania, Italy\\
\em marano@dmi.unict.it\\
[\bigskipamount]
\bf Kanishka Perera\thanks{This work was completed partly while the fourth-named author was visiting Universit\`{a} di Reggio Calabria and partly while the second-named author was visiting the Florida Institute of Technology, and they are grateful for the kind hospitality of the host universities.}\\
Florida Institute of Technology\\
Melbourne, FL 32901, USA\\
\em kperera@fit.edu}
\date{}

\begin{document}

\maketitle

\begin{abstract}
We prove the existence of $N$ distinct pairs of nontrivial solutions for critical $p$-Laplacian problems in $\R^N$, as well as in bounded domains. To overcome the difficulties arising from the lack of compactness, we use a recent global compactness result of Mercuri and Willem.
\end{abstract}

\section{Introduction}

Consider the problem
\begin{equation} \label{1.1}
- \Delta_p\, u = a(x)\, |u|^{p-2}\, u + |u|^{p^\ast - 2}\, u, \quad u \in \D^{1,p}(\R^N),
\end{equation}
where $1 < p < N$, $p^\ast = Np/(N - p)$ is the critical Sobolev exponent, and $a \in L^{N/p}(\R^N)$ satisfies
\begin{equation} \label{1.2}
\inf_{u \in \D^{1,p}(\R^N),\; \pnorm{\nabla u} = 1}\, \int_{\R^N} \big(|\nabla u|^p - a(x)\, |u|^p\big)\, dx > 0
\end{equation}
and
\begin{equation} \label{1.3}
\essinf_{x \in B_\delta(x_0)}\, a(x) > 0
\end{equation}
for some open ball $B_\delta(x_0) \subset \R^N$. Here $\pnorm[q]{\cdot}$ denotes the norm in $L^q(\R^N)$. Let
\begin{equation} \label{1.4}
S = \inf_{u \in \D^{1,p}(\R^N) \setminus \set{0}}\, \frac{\dint_{\R^N} |\nabla u|^p\, dx}{\left(\dint_{\R^N} |u|^{p^\ast}\, dx\right)^{p/p^\ast}}
\end{equation}
be the best constant in the Sobolev inequality. We prove the following multiplicity result.

\begin{theorem} \label{Theorem 1.1}
Assume that $N \ge p^2$, and \eqref{1.2} and \eqref{1.3} hold. If
\begin{equation} \label{1.5}
\spnorm[N/p]{a^+} < \big(1 - 2^{-p/N}\big)\, S,
\end{equation}
where $a^+$ is the positive part of $a$ defined by $a^+(x) = \max \set{a(x),0}$, then problem \eqref{1.1} has at least $N$ pairs of nontrivial solutions.
\end{theorem}

Sufficient conditions for the existence of a positive solution of problem \eqref{1.1} when $p \ge 2$ and $a(x) \le 0$ for all $x \in \R^N$ were given by Alves \cite{MR1926623}. Related multiplicity results for the subcritical scalar field equation in $\R^N$ can be found in Clapp and Weth \cite{MR2103845} and Perera \cite{MR3259012}.

We prove the multiplicity result in Theorem \ref{Theorem 1.1} in bounded domains also (see Theorem \ref{Theorem 3.1}). In particular, we have the following corollary for the problem
\begin{equation} \label{1.6}
- \Delta_p\, u = \lambda\, |u|^{p-2}\, u + |u|^{p^\ast - 2}\, u, \quad u \in W^{1,p}_0(\Omega),
\end{equation}
where $\Omega$ is a smooth bounded domain in $\R^N$ and $1 < p < N$.

\begin{corollary} \label{Corollary 1.2}
If $N \ge p^2$ and
\begin{equation} \label{1.7}
0 < \lambda < \big(1 - 2^{-p/N}\big)\, S \vol{\Omega}^{-p/N},
\end{equation}
where $\vol{\Omega}$ is the Lebesgue measure of $\Omega$, then problem \eqref{1.6} has at least $N$ pairs of nontrivial solutions.
\end{corollary}

See Devillanova and Solimini \cite{MR1966256} for the semilinear case $p = 2$. For $N > p^2 + p$, Cao et al.\! \cite{MR2885967} have recently shown that problem \eqref{1.6} has infinitely many solutions for all $\lambda > 0$ (see also Wu and Huang \cite{MR3072257} and Perera et al.\! \cite{PeSqYa1}). The multiplicity result in Corollary \ref{Corollary 1.2} is new when $p \ne 2$ and $p^2 \le N \le p^2 + p$.

From a technical point of view, we first use a recent global compactness result of Mercuri and Willen \cite{MR2644751} to prove that suitable Palais-Smale sequences of the energy functional associated with problem (\ref{1.1}) are weakly compact; see Lemma \ref{Lemma 2.2}. Estimates of critical levels arising from Krasnoselskii's genus are then provided; see Lemmas \ref{Lemma 2.3}--\ref{Lemma 2.6}. In both cases, the so called Talenti's functions \cite{MR0463908} play a basic role.

\section{Proof of Theorem \ref{Theorem 1.1}}

Weak solutions of problem \eqref{1.1} coincide with critical points of the $C^1$-functional
\[
\Phi(u) = \int_{\R^N} \left[\frac{1}{p}\, \big(|\nabla u|^p - a(x)\, |u|^p\big) - \frac{1}{p^\ast}\, |u|^{p^\ast}\right] dx, \quad u \in \D^{1,p}(\R^N)
\]
restricted to the Nehari manifold
\[
\M = \set{u \in \D^{1,p}(\R^N) \setminus \set{0} : \int_{\R^N} \big(|\nabla u|^p - a(x)\, |u|^p\big)\, dx = \int_{\R^N} |u|^{p^\ast}\, dx}.
\]
We note that $\M$ is bounded away from the origin in $\D^{1,p}(\R^N)$ in view of \eqref{1.2} and the continuity of the Sobolev imbedding $\D^{1,p}(\R^N) \hookrightarrow L^{p^\ast}(\R^N)$.

Recall that for $c \in \R$, $\seq{u_j} \subset \M$ is a \PS{c} sequence for $\Phi$ (resp.\! $\restr{\Phi}{\M}$) if $\Phi'(u_j) \to 0$ (resp.\! $\restr{\Phi}{\M}'(u_j) \to 0$) and $\Phi(u_j) \to c$.

\begin{lemma} \label{Lemma 2.1}
If $\seq{u_j} \subset \M$ is a {\em \PS{c}} sequence for $\restr{\Phi}{\M}$, then $\seq{u_j}$ is also a {\em \PS{c}} sequence for $\Phi$.
\end{lemma}

\begin{proof}
Since $u_j \in \M$,
\[
\Phi(u_j) = \frac{1}{N} \int_{\R^N} \big(|\nabla u_j|^p - a(x)\, |u_j|^p\big)\, dx,
\]
and since $\seq{\Phi(u_j)}$ is bounded and \eqref{1.2} holds, this implies that $\seq{u_j}$ is bounded in $\D^{1,p}(\R^N)$. Since $\restr{\Phi}{\M}'(u_j) \to 0$, for some sequence of Lagrange multipliers $\seq{\mu_j} \subset \R$,
\begin{multline} \label{2.1}
\int_{\R^N} \big(|\nabla u_j|^{p-2}\, \nabla u_j \cdot \nabla v - a(x)\, |u_j|^{p-2}\, u_j\, v - |u_j|^{p^\ast - 2}\, u_j\, v\big)\, dx\\[10pt]
- \mu_j \int_{\R^N} \Big[p\, \big(|\nabla u_j|^{p-2}\, \nabla u_j \cdot \nabla v - a(x)\, |u_j|^{p-2}\, u_j\, v\big) - p^\ast\, |u_j|^{p^\ast - 2}\, u_j\, v\Big]\, dx\\[10pt]
= \o(\norm{v}) \quad \forall v \in \D^{1,p}(\R^N).
\end{multline}
Taking $v = u_j$ and using the fact that $\seq{u_j}$ is a bounded sequence in $\M$ shows that
\[
\mu_j \int_{\R^N} \big(|\nabla u_j|^p - a(x)\, |u_j|^p\big)\, dx \to 0.
\]
Since \eqref{1.2} holds and $\M$ is bounded away from the origin, this implies that $\mu_j \to 0$, so $\Phi'(u_j) \to 0$ by \eqref{2.1}.
\end{proof}

In the absence of a compact Sobolev imbedding, we will use the following compactness type property of \PS{c} sequences for $\restr{\Phi}{\M}$.

\begin{lemma} \label{Lemma 2.2}
If
\begin{equation} \label{2.2}
\frac{1}{N}\, S^{N/p} < c < \frac{2}{N}\, S^{N/p},
\end{equation}
then every {\em \PS{c}} sequence for $\restr{\Phi}{\M}$ has a subsequence that converges weakly to a (nontrivial) solution $v_0$ of problem \eqref{1.1} with $\Phi(v_0) = c$ or $\Phi(v_0) = c - \dfrac{1}{N}\, S^{N/p}$.
\end{lemma}

\begin{proof}
Assume \eqref{2.2} and let $\seq{u_j} \subset \M$ be a \PS{c} sequence for $\restr{\Phi}{\M}$. Then $\seq{u_j}$ is also a \PS{c} sequence for $\Phi$ by Lemma \ref{Lemma 2.1}, and by a recent global compactness result of Mercuri and Willem \cite[Theorem 5.2]{MR2644751} (see also Benci and Cerami \cite{MR1033915} and Alves \cite{MR1926623}), a renamed subsequence converges weakly to a solution $v_0 \in \D^{1,p}(\R^N)$ of problem \eqref{1.1} and there exist nontrivial solutions $v_1,\dots,v_k \in \D^{1,p}(\R^N)$, $k \ge 0$, of $- \Delta_p\, u = |u|^{p^\ast - 2}\, u$ such that
\begin{equation} \label{2.3}
\Phi(v_0) + \sum_{i=1}^k \Phi_\infty(v_i) = c,
\end{equation}
where
\[
\Phi_\infty(u) = \int_{\R^N} \left(\frac{1}{p}\, |\nabla u|^p - \frac{1}{p^\ast}\, |u|^{p^\ast}\right) dx, \quad u \in \D^{1,p}(\R^N).
\]
If $k = 0$, then $\Phi(v_0) = c$ by \eqref{2.3} and we are done, so suppose $k \ge 1$. We have
\begin{equation} \label{2.4}
\Phi(v_0) = \frac{1}{N} \int_{\R^N} |v_0|^{p^\ast}\, dx \ge 0, \qquad \Phi_\infty(v_i) = \frac{1}{N} \int_{\R^N} |v_i|^{p^\ast}\, dx > 0, \quad i = 1,\dots,k.
\end{equation}
For $i = 1,\dots,k$, if $v_i$ is sign-changing, then $\Phi_\infty(v_i) \ge \dfrac{2}{N}\, S^{N/p}$ (see, e.g., Mercuri et al.\! \cite[Lemma 2.1]{MR3250483}), so $|v_i| > 0$ by \eqref{2.2}--\eqref{2.4} and the strong maximum principle. Hence
\[
\Phi_\infty(v_i) = \frac{1}{N}\, S^{N/p}, \quad k = 1,\dots,k
\]
by Sciunzi \cite[Theorem 1.1]{Sc} (see also Caffarelli et al.\! \cite{MR982351}, Damascelli et al.\! \cite{MR3255462}, and V{\'e}tois \cite{MR3411668}), and then it follows from \eqref{2.2}--\eqref{2.4} again that $k = 1$ and $\Phi(v_0) = c - \dfrac{1}{N}\, S^{N/p}$.
\end{proof}

Let $\A$ denote the class of all nonempty compact symmetric subsets of $\M$, let
\[
\gamma(A) = \inf\, \{k \ge 1 : \exists \text{ an odd continuous map } A \to \R^k \setminus \{0\}\}
\]
be the Krasnoselskii genus of $A \in \A$, let
\[
\A_k = \set{A \in \A : \gamma(A) \ge k + 1},
\]
and set
\[
c_k := \inf_{A \in \A_k}\, \max_{u \in A}\, \Phi(u), \quad k = 0,\dots,N.
\]
We have
\begin{equation} \label{2.5}
c_0 = \inf_{u \in \M}\, \Phi(u)
\end{equation}
and
\[
c_0 \le c_1 \le \cdots \le c_N,
\]
and we will show that $\dfrac{1}{N}\, S^{N/p} < c_1$ and $c_N < \dfrac{2}{N}\, S^{N/p}$ in order to apply Lemma \ref{Lemma 2.2}.

\begin{lemma} \label{Lemma 2.3}
Every $A \in \A_1$ contains a point $u_0$ such that $u_0^\pm \in \M$, and hence $c_1 \ge 2 c_0$.
\end{lemma}

\begin{proof}
Let $A \in \A_1$ and set $\alpha(u) = \dint_{\R^N} \big(|\nabla u|^p - a(x)\, |u|^p - |u|^{p^\ast}\big)\, dx$. Then $A$ contains a point $u_0$ with $\alpha(u_0^+) = \alpha(u_0^-)$, for otherwise
\[
A \to \R \setminus \{0\}, \quad u \mapsto \frac{\alpha(u^+) - \alpha(u^-)}{|\alpha(u^+) - \alpha(u^-)|}
\]
is an odd continuous map and hence $\gamma(A) \le 1$. We also have $\alpha(u_0^+) + \alpha(u_0^-) = \alpha(u_0) = 0$ since $u_0 \in \M$, so this implies that $\alpha(u_0^\pm) = 0$ and hence $u_0^\pm \in \M$. Then
\[
\max_{u \in A}\, \Phi(u) \ge \Phi(u_0) = \Phi(u_0^+) + \Phi(u_0^-) \ge 2 c_0
\]
by \eqref{2.5}, so the second assertion follows.
\end{proof}

We have the following lower bound for $c_0$.

\begin{lemma} \label{Lemma 2.4}
If $\spnorm[N/p]{a^+} < S$, then
\[
c_0 \ge \frac{1}{N}\, \big(S - \spnorm[N/p]{a^+}\big)^{N/p}.
\]
\end{lemma}

\begin{proof}
For all $u \in \M$,
\begin{multline*}
S \left(\int_{\R^N} |u|^{p^\ast}\, dx\right)^{p/p^\ast} \le \int_{\R^N} |\nabla u|^p\, dx = \int_{\R^N} \big(a(x)\, |u|^p + |u|^{p^\ast}\big)\, dx\\[10pt]
\le \int_{\R^N} \big(a^+(x)\, |u|^p + |u|^{p^\ast}\big)\, dx \le \spnorm[N/p]{a^+} \left(\int_{\R^N} |u|^{p^\ast}\, dx\right)^{p/p^\ast} + \int_{\R^N} |u|^{p^\ast}\, dx
\end{multline*}
by \eqref{1.4} and the H\"{o}lder inequality, so
\[
\int_{\R^N} |u|^{p^\ast}\, dx \ge \big(S - \spnorm[N/p]{a^+}\big)^{N/p}.
\]
Since
\[
\Phi(u) = \frac{1}{N} \int_{\R^N} |u|^{p^\ast}\, dx,
\]
the assertion follows.
\end{proof}

\begin{lemma} \label{Lemma 2.5}
$c_1 > \dfrac{1}{N}\, S^{N/p}$
\end{lemma}

\begin{proof}
By Lemmas \ref{Lemma 2.3} and \ref{Lemma 2.4}, and \eqref{1.5},
\[
c_1 \ge \frac{2}{N}\, \big(S - \spnorm[N/p]{a^+}\big)^{N/p} > \frac{1}{N}\, S^{N/p}. \QED
\]
\end{proof}

\begin{lemma} \label{Lemma 2.6}
$c_N < \dfrac{2}{N}\, S^{N/p}$
\end{lemma}

\begin{proof}
The infimum in \eqref{1.4} is attained by the family of functions
\[
u_\eps(x) = \frac{C_{N,p}\, \eps^{-(N-p)/p}}{\left[1 + \left(\dfrac{|x|}{\eps}\right)^{p/(p-1)}\right]^{(N-p)/p}}, \quad \eps > 0,
\]
where $C_{N,p} > 0$ is chosen so that
\[
\int_{\R^N} |\nabla u_\eps|^p\, dx = \int_{\R^N} u_\eps^{p^\ast}\, dx = S^{N/p}.
\]
Take a smooth function $\eta : [0,\infty) \to [0,1]$ such that $\eta(s) = 1$ for $s \le 1/4$ and $\eta(s) = 0$ for $s \ge 1/2$, and set
\[
\tilde{u}_\eps(x) = \eta(|x|)\, u_\eps(x), \quad \eps > 0.
\]
We have the well-known estimates
\begin{gather}
\label{2.6} \int_{\R^N} |\nabla \tilde{u}_\eps|^p\, dx \le S^{N/p} + C \eps^{(N-p)/(p-1)},\\[10pt]
\int_{\R^N} \tilde{u}_\eps^{p^\ast}\, dx \ge S^{N/p} - C \eps^{N/(p-1)},\\[10pt]
\label{2.8} \int_{\R^N} \tilde{u}_\eps^p\, dx \ge \begin{cases}
\dfrac{\eps^p}{C} & \text{if } N > p^2\\[10pt]
\dfrac{\eps^p}{C}\; |\!\log \eps| & \text{if } N = p^2,
\end{cases}
\end{gather}
where $C = C(N,p) > 0$ is a constant (see, e.g., Degiovanni and Lancelotti \cite{MR2514055}).

After a translation and a dilation, we may assume that $x_0 = 0$ and $\delta = 1$ in \eqref{1.3}, so we have
\begin{equation} \label{2.9}
\lambda := \essinf_{x \in B_1(0)}\, a(x) > 0.
\end{equation}
Let $S^{N-1}$ be the unit sphere in $\R^N$, let
\[
S^N_+ = \bgset{x = (x' \sqrt{1 - t^2},t) : x' \in S^{N-1},\, t \in [0,1]}
\]
be the upper hemisphere in $R^{N+1}$, and consider the map $\varphi : S^N_+ \to \M$ defined by
\[
\varphi(x) = \pi\big(\tilde{u}_\eps(\cdot - (1 - (2t - 1)_+)\, x'/2) - (1 - 2t)_+\, \tilde{u}_\eps(\cdot + x'/2)\big),
\]
where
\[
\pi : \D^{1,p}(\R^N) \setminus \set{0} \to \M, \quad u \mapsto \left[\frac{\dint_{\R^N} \big(|\nabla u|^p - a(x)\, |u|^p\big)\, dx}{\dint_{\R^N} |u|^{p^\ast}\, dx}\right]^{N/p p^\ast} u
\]
is the radial projection onto $\M$. Clearly, $\varphi$ is continuous. Since
\[
\Phi(\pi(u)) = \frac{1}{N} \left[\frac{\dint_{\R^N} \big(|\nabla u|^p - a(x)\, |u|^p\big)\, dx}{\left(\dint_{\R^N} |u|^{p^\ast}\, dx\right)^{p/p^\ast}}\right]^{N/p}, \quad u \in \D^{1,p}(\R^N) \setminus \set{0},
\]
and $\tilde{u}_\eps(\cdot - (1 - (2t - 1)_+)\, x'/2)$ and $\tilde{u}_\eps(\cdot + x'/2)$ have disjoint supports in $B_1(0)$, where $a \ge \lambda$ a.e.\! by \eqref{2.9},
\[
\Phi(\varphi(x)) \le \frac{1}{N} \left[\frac{\left(1 + (1 - 2t)_+^p\right) \dint_{\R^N} \big(|\nabla \tilde{u}_\eps|^p - \lambda\, \tilde{u}_\eps^p\big)\, dx}{\big(1 + (1 - 2t)_+^{p^\ast}\big)^{p/p^\ast} \left(\dint_{\R^N} \tilde{u}_\eps^{p^\ast}\, dx\right)^{p/p^\ast}}\right]^{N/p} \quad \forall x \in S^N_+.
\]
The right-hand side is nonincreasing in $t$, and \eqref{2.6}--\eqref{2.8} give
\[
\frac{\dint_{\R^N} \big(|\nabla \tilde{u}_\eps|^p - \lambda\, \tilde{u}_\eps^p\big)\, dx}{\left(\dint_{\R^N} \tilde{u}_\eps^{p^\ast}\, dx\right)^{p/p^\ast}} \le \begin{cases}
S - \dfrac{\lambda \eps^p}{C} + C \eps^{(N-p)/(p-1)} & \text{if } N > p^2\\[10pt]
S - \dfrac{\lambda \eps^p}{C}\; |\!\log \eps| + C \eps^p & \text{if } N = p^2,
\end{cases}
\]
so
\[
\max_{x \in S^N_+}\, \Phi(\varphi(x)) < \frac{2}{N}\, S^{N/p}
\]
if $\eps$ is sufficiently small. Since $\varphi$ is odd on $S^{N-1}$, it can now be extended to an odd continuous map $\tilde{\varphi} : S^N \to \M$ satisfying
\[
\max_{u \in \tilde{\varphi}(S^N)}\, \Phi(u) < \frac{2}{N}\, S^{N/p}.
\]
Then
\[
\gamma(\tilde{\varphi}(S^N)) \ge \gamma(S^N) = N + 1
\]
and the assertion follows.
\end{proof}

The next lemma is due to Devillanova and Solimini \cite{MR1966256} when $p = 2$.

\begin{lemma} \label{Lemma 2.7}
If $c_k = c_{k+1}$ for some $k \in \set{1,...,N - 1}$, then $\Phi$ has infinitely many critical points at the level $c_k$ or $c_k - \dfrac{1}{N}\, S^{N/p}$.
\end{lemma}

\begin{proof}
Suppose $\Phi$ has only finitely many critical points $v_1,\dots,v_m$ at the levels $c_k$ and $c_k - \dfrac{1}{N}\, S^{N/p}$, and let $\set{w_1,\dots,w_n}$ be a basis for their span. We have
\[
v_i = \sum_{j=1}^n a_{ij}\, w_j, \quad i = 1,\dots,m
\]
for some $a_{ij} \in \R$. Take $(b_1,\dots,b_n) \in \R^n$ such that
\[
\sum_{j=1}^n a_{ij}\, b_j \ne 0, \quad i = 1,\dots,m,
\]
and let $l$ be a bounded linear functional on $D^{1,p}(\R^N)$ such that $l(w_j) = b_j,\, j = 1,\dots,n$. Then
\begin{equation} \label{2.11}
l(v_i) = \sum_{j=1}^n a_{ij}\, l(w_j) = \sum_{j=1}^n a_{ij}\, b_j \ne 0, \quad i = 1,\dots,m.
\end{equation}

Now take a sequence $\seq{A_j} \subset \A_{k+1}$ such that $\max \Phi(A_j) \to c_{k+1}$, and let
\[
\tilde{A}_j = \bgset{u \in A_j : l(u) = 0}.
\]
By the monotonicity of the genus,
\[
k + 2 \le \gamma(A_j) \le \gamma(\tilde{A}_j) + \gamma(A_j \setminus \tilde{A}_j),
\]
and $\gamma(A_j \setminus \tilde{A}_j) \le 1$ since $\restr{l}{A_j \setminus \tilde{A}_j}$ is an odd continuous mapping into $\R \setminus \set{0}$, so $\gamma(\tilde{A}_j) \ge k + 1$ and hence $\tilde{A}_j \in \A_k$. Then
\[
c_k \le \max_{u \in \tilde{A}_j}\, \Phi(u) \le \max_{u \in A_j}\, \Phi(u) \to c_{k+1} = c_k,
\]
so $\max \Phi(\tilde{A}_j) \to c_k$. By Ghoussoub \cite[Theorem 1]{MR1103905}, then $\restr{\Phi}{\M}$ has a \PS{c_k} sequence $\seq{u_j}$ such that
\begin{equation} \label{2.12}
\dist{u_j}{\tilde{A}_j} \to 0.
\end{equation}
Since $\dfrac{1}{N}\, S^{N/p} < c_k < \dfrac{2}{N}\, S^{N/p}$ by Lemmas \ref{Lemma 2.5} and \ref{Lemma 2.6}, then a renamed subsequence converges weakly to some $v_i$ by Lemma \ref{Lemma 2.2}. Then \eqref{2.12} implies that $l(u_j) \to 0$ and hence $l(v_i) = 0$, contradicting \eqref{2.11}.
\end{proof}

We are now ready to prove Theorem \ref{Theorem 1.1}.

\begin{proof}[Proof of Theorem \ref{Theorem 1.1}]
By Lemmas \ref{Lemma 2.5} and \ref{Lemma 2.6},
\[
\frac{1}{N}\, S^{N/p} < c_1 \le \cdots \le c_N < \frac{2}{N}\, S^{N/p},
\]
and hence $c_k$ or $c_k - \dfrac{1}{N}\, S^{N/p}$ is a critical level of $\Phi$ for $k = 1,\dots,N$ by Lemma \ref{Lemma 2.2}. If $c_k = c_{k+1}$ for some $k \in \set{1,...,N - 1}$, then $\Phi$ has infinitely many critical points by Lemma \ref{Lemma 2.7} and we are done, so suppose that this is not the case. Then
\[
c_1 - \frac{1}{N}\, S^{N/p} < \cdots < c_N - \frac{1}{N}\, S^{N/p} < c_1 < \cdots < c_N
\]
and at least $N$ of these levels are critical for $\Phi$.
\end{proof}

\section{Bounded domains}

Consider the problem
\begin{equation} \label{3.1}
- \Delta_p\, u = a(x)\, |u|^{p-2}\, u + |u|^{p^\ast - 2}\, u, \quad u \in W^{1,p}_0(\Omega),
\end{equation}
where $\Omega$ is a smooth bounded domain in $\R^N$, $1 < p < N$, $a \in L^{N/p}(\Omega)$ satisfies
\begin{equation} \label{3.2}
\inf_{u \in W^{1,p}_0(\Omega),\; \pnorm{\nabla u} = 1}\, \int_\Omega \big(|\nabla u|^p - a(x)\, |u|^p\big)\, dx > 0
\end{equation}
and
\begin{equation} \label{3.3}
\essinf_{x \in B_\delta(x_0)}\, a(x) > 0
\end{equation}
for some open ball $B_\delta(x_0) \subset \Omega$, and $\pnorm[q]{\cdot}$ denotes the norm in $L^q(\Omega)$.

\begin{theorem} \label{Theorem 3.1}
Assume that $N \ge p^2$, and \eqref{3.2} and \eqref{3.3} hold. If
\[
\spnorm[N/p]{a^+} < \big(1 - 2^{-p/N}\big)\, S,
\]
then problem \eqref{3.1} has at least $N$ pairs of nontrivial solutions.
\end{theorem}

\begin{proof}
We argue as in the proof of Theorem \ref{Theorem 1.1}. Lemma \ref{Lemma 2.2} is now proved using the variant global compactness result in Proposition \ref{Proposition 3.2} below, which readily follows by arguing as in Mercuri and Willem \cite{MR2644751} (see also Mercuri et al.\! \cite[Remark 2.1]{MR3250483}). The rest of the proof is unchanged.
\end{proof}

\begin{proposition} \label{Proposition 3.2}
Let $\seq{u_j} \subset W^{1,p}_0(\Omega)$ be a {\em \PS{c}} sequence for the functional
\[
\Phi(u) = \int_\Omega \left[\frac{1}{p}\, \big(|\nabla u|^p - a(x)\, |u|^p\big) - \frac{1}{p^\ast}\, |u|^{p^\ast}\right] dx, \quad u \in W^{1,p}_0(\Omega).
\]
Then a renamed subsequence converges weakly to a (possibly trivial) solution $v_0 \in W^{1,p}_0(\Omega)$ of problem \eqref{3.1}, and there exist nontrivial solutions $v_1,\dots,v_k \in \D^{1,p}(\R^N)$ or (up to a rotation and a translation) $\D^{1,p}_0(\R^N_+)$, $k \ge 0$, of $- \Delta_p\, u = |u|^{p^\ast - 2}\, u$ and sequences $(y^i_j) \subset \closure{\Omega},\, (\lambda^i_j) \subset \R_+,\, i = 1,\dots,k$ such that $\dist{y^i_j}{\bdry{\Omega}}/\lambda^i_j \to \infty$ in the case of $\R^N$ and $\dist{y^i_j}{\bdry{\Omega}}/\lambda^i_j$ is bounded in the case of $\R^N_+$,
\begin{gather*}
\norm{u_j - v_0 - \sum_{i=1}^k (\lambda^i_j)^{-(N-p)/p}\, v_i((\cdot - y^i_j)/\lambda^i_j)} \to 0,\\[10pt]
\norm{u_j}^p \to \sum_{i=0}^k \norm{v_i}^p,\\[10pt]
\Phi(v_0) + \sum_{i=1}^k \Phi_\infty(v_i) = c,
\end{gather*}
where $\Phi_\infty$ is defined on $\D^{1,p}(\R^N)$ and $\D^{1,p}_0(\R^N_+)$ as in the proof of Lemma \ref{Lemma 2.2}.
\end{proposition}

Finally, to see that Corollary \ref{Corollary 1.2} in the introduction follows from Theorem \ref{Theorem 3.1}, we note that for all $u \in W^{1,p}_0(\Omega)$ with $\pnorm{\nabla u} = 1$,
\[
\int_\Omega \big(|\nabla u|^p - \lambda\, |u|^p\big)\, dx \ge 1 - \lambda \left(\int_\Omega |u|^{p^\ast}\, dx\right)^{p/p^\ast} \vol{\Omega}^{p/N} \ge 1 - \lambda\, S^{-1} \vol{\Omega}^{p/N} > 2^{-p/N}
\]
by the H\"{o}lder inequality, \eqref{1.4}, and \eqref{1.7}.

\section*{Acknowledgements}
\noindent The authors have been partially supported by the Gruppo Nazionale per l'Analisi Matematica, la Probabilit\`{a} e le loro Applicazioni (GNAMPA) of the Istituto Nazionale di Alta Matematica (INdAM).

\end{document}